\documentclass[11pt]{amsart}
\usepackage{amsfonts}
\usepackage{amsmath}
\usepackage{amsthm}
\usepackage{url}
\usepackage[all]{xy}
\usepackage{latexsym}
\usepackage{amssymb}
\usepackage[cp850]{inputenc}
\usepackage{amsfonts}
\usepackage{graphicx}

\usepackage[usenames]{color}

\newtheorem{sat}{Theorem}[section]		
\newtheorem{lem}[sat]{Lemma}
\newtheorem{kor}[sat]{Corollary}			
\newtheorem{prop}[sat]{Proposition}

\newtheorem*{defi*}{Definition}			
\newtheorem*{bei*}{Example}
\newtheorem*{sat*}{Theorem}				
\newtheorem*{kor*}{Corollary}
\newtheorem*{rmk*}{Remark}				
	
\newtheorem*{quest*}{Question}

\let\ssection=\section
\renewcommand{\section}{\setcounter{equation}{0}\ssection}

\newtheorem*{namedtheorem}{\theoremname}
\newcommand{\theoremname}{testing}

\theoremstyle{remark}
\newtheorem*{bem}{Remark}

\newtheorem*{namedtheoremr}{\theoremnamer}
\newcommand{\theoremnamer}{testing}

\newcommand{\BR}{\mathbb R}			
			\newcommand{\BQ}{\mathbb Q}
\newcommand{\BS}{\mathbb S}			\newcommand{\BZ}{\mathbb Z}

\newcommand{\D}{\partial}
\newcommand{\DD}{\nabla}

\newcommand{\comment}[1]{}

\newcommand{\fsubd}{\mathrel{{\scriptstyle\searrow}\kern-1ex^d\kern0.5ex}}
\newcommand{\bsubd}{\mathrel{{\scriptstyle\swarrow}\kern-1.6ex^d\kern0.8ex}}

\begin{document}

\title[]{A remark about critical sets in $\BR^3$}
\author{Juan Souto}
\address{IRMAR, Universit\'e de Rennes 1}
\email{juan.souto@univ-rennes1.fr}
\thanks{This material is based upon work supported by the National Science Foundation under Grant No. DMS-1440140 while the author was in residence at the Mathematical Sciences Research Institute in Berkeley, California, during the Fall 2016 semester.}
\begin{abstract}
We give a necessary condition for a closed subset of $\BR^3$ to be the set of critical points of some smooth function. In particular we obtain that for example neither the Whitehead continuum nor the p-adic solenoid are such a critical sets.
\end{abstract}
\maketitle

\section{}
It is classical, and otherwise easy to see, that every closed set $X\subset\BR^n$ is the vanishing set $\{f=0\}$ of some smooth function $f\in C^\infty(\BR^n)$. On the other hand, we are far from understanding which sets are critical. Here a set $X\subset\BR^n$ is {\em critical} if there is a smooth function $f\in C^\infty(\BR^n)$ with $X$ as its set of criticial points, that is $X=\{df=0\}$. In fact, while it is easy to determine if a submanifold $X\subset\BR^n$ is a critical set, more general critical sets are only fully understood if the dimension of the ambient space is very small, namely $1$ or $2$ (see \cite{Norton-Pugh}). 

In \cite{Grayson-Pugh}, the paper motivating our work, Grayson and Pugh prove that a wealth of rather wild closed subsets of $\BR^3$ are critical. They prove for example that Antoine's necklace (Theorem A) is critical. Actually, as they mention in page 26 of their paper, their argument can be modified to prove that every Cantor set in $\BR^3$ is also critical. The Denjoy continuum (or for that matter, any proper closed subset of a smooth connected subsurface of $\BR^3$) is also critical (Theorem 2.21). Also, any closed set $C\subset\BR^3$ with $\BR^3\setminus C\simeq\BS^2\times\BR$ is critical as well (Corollary 2.6). The image $\phi(C)$ of a critical set $C\subset\BR^3$ under a homeomorphism $\phi:\BR^3\to\BR^3$ is also critical (Theorem 2.5). And so on...

The aim of this note is not to construct even more critical sets, but to give an obstruction for a set to be critical. Our main result is the following:

\begin{sat}\label{sat1}
Let $X\subset\BR^3$ be a closed and connected set whose complement $\BR^3\setminus X$ is connected and satisfies $\dim_\BR H_1(\BR^3\setminus X;\BR)<\infty.$ If $X$ is critical, then $\BR^3\setminus X$ is tame.
\end{sat}

Recall that an open 3-manifold is {\em tame} if it is homeomorphic to the interior of a compact manifold with boundary. 

Maybe, the most prominent closed connected set $X\subset\BR^3$ with non-tame complement is the Whitehead continuum \cite{Rolfsen}. In fact, the complement of the Whitehead continuum is simply connected, which means that Theorem \ref{sat1} applies to it:

\begin{kor}
The Whitehead continuum is not critical.\qed
\end{kor}

Recall now that, as an abstract topological space, a {\em solenoid} is a topological space obtained as the inverse limit of circles. More concretely, if we are given any sequence $(n_i)$ of integers $n_i\ge 2$ and consider the maps
$$\mu_i:\BS^1\to\BS^1,\ \ \mu_i(z)=z^{n_i}$$
then the associated solenoid is the inverse limit 
$$X=\varprojlim(\BS^1\stackrel{\mu_1}\longleftarrow\BS^1\stackrel{\mu_2}\longleftarrow\BS^1\stackrel{\mu_3}\longleftarrow\dots)$$
All solenoids can be embedded in $\BR^3$ and, abusing terminology we will say that a closed set $X\subset\BR^3$ is a solenoid if it is abstractly homeomorphic to one. We get from Lemma \ref{lem-mladen} below that solenoid complements are not tame but have finitely generated first homology with coefficients in $\BR$. It follows thus from Theorem \ref{sat1} that solenoids are not critical:

\begin{kor}\label{kor-solenoid}
Solenoids in $\BR^3$ are not critical.\qed
\end{kor}

Corollary \ref{kor-solenoid} gives a positive answer to a conjecture of Grayson-Pugh \cite[p. 26]{Grayson-Pugh} asserting that the p-adic solenoid is not critical. 

\begin{bem}
We stress that Corollary \ref{kor-solenoid} applies to every subset of $\BR^3$ which is homeomorphic to a solenoid. This is is somewhat exceptional. Or more precisely, this fails to be true for the Whiteahead continuum. In fact, there is $X\subset\BR^3$ homeomorphic to the Whitehead continuum and whose complement $\BR^3\setminus X$ is homeomorphic to $\BS^2\times\BR$. This last property implies that $X$ is a critical set \cite[Corollary 2.6]{Grayson-Pugh}.
\end{bem}

Before concluding the introduction we should mention that, although all our work is topological, the question of knowing which sets are critical has deep connections with dynamics. For example, as discussed in \cite{Grayson-Pugh}, a closed set $X\subset\BR^n$ is critical if and only if it is the set of chain recurrent points of a smooth flow $(\phi_t)$ in $\BR^n$. In fact, not only do critical sets have such a dynamical interpretation, but also many of the sets of interest are of dynamical nature. For example, p-adic solenoids arise as so called Smale-Williams attractors. There are however other very prominent attractors such as the Lorenz attractor, which we suspect to actually be critical.
\medskip

\noindent{\bf Acknowledgements.} We thank Dominique Cerveau, Mladen Bestvina, and especially Pekka Pankka for very interesting conversations on the topic of this note.

\section{}
We discuss next a few facts about 3-manifolds, but first we refer to Rolfsen's book \cite{Rolfsen} for general facts and definitions. We also refer to \cite{McMillan} for examples of non-tame manifolds. Having done that, we start by recalling the following result by Tucker \cite{Tucker}:

\begin{sat*}[Tucker]
An open $\BR P^2$-irreducible 3-manifold $M$ is tame if and only if for every compact submanifold $C\subset M$ and for every connected component $U\in\pi_0(M\setminus C)$ we have that $\pi_1(U)$ is finitely generated.
\end{sat*}

Recall that a 3-manifold is $\BR P^2$-irreducible if it does not contain embedded projective planes and if every embedded (smooth) 2-sphere bounds a ball. All the manifolds in this note will arise as complements $\BR^3\setminus X$ of closed connected subsets of $\BR^3$. In this particular situation we get the following from Tucker's theorem:

\begin{lem}\label{lem-tucker}
Suppose that $X\subset\BR^3$ is a closed connected set with connected complement. Then $\BR^3\setminus X$ is tame if and only if $\pi_1(N\setminus X)$ is finitely generated for every compact submanifold $N\subset\BR^3$ which contains $X$ in its interior.
\end{lem}

Before launching the proof note that every closed connected set $X\subset\BR^3$ with connected complement can be obtained as the intersection $X=\cap_{i=1}^\infty C_i$ of a nested sequence of compact 3-dimensional submanifolds $C_1\supset C_2\supset C_3\supset\dots$ with connected boundary. We refer to any such sequence as {\em a defining sequence} for $X$. Note that every subsequence of a defining sequence for $X$ is a defining sequence for $X$ in its own right.

\begin{proof}
If $\BR^3\setminus X$ is tame then the claim clearly holds because $N\setminus X$ is also tame for any $N$ as in the statement. To prove the other direction let $C_1\supset C_2\supset C_3\supset\dots$ be a defining sequence for $X$, and suppose for a moment that $C_i\setminus X$ fails to be irreducible for infinitely many, say for all, $i$. Since $\BR^3$ does not contain a copy of $\BR P^2$, it follows that for each $i$ the manifold $C_i\setminus X$ contains some 2-sphere $S_i$ which does not bound a ball. Note also that up to passing to a subsequence we might assume that $S_i\cap C_{i+1}=\emptyset$ for all $i$. Considering now $X$ a subset of $\BS^3=\BR^3\cup\{\infty\}$ note that the sphere $S_i$ separates $\BS^3$ into two balls $B_i,B_i'$. The set $X$ being connected and disjoint of $S_i$, we get that it is contained in one of these balls. Say that $X\subset B_i'$ for all $i$. Noting that $B_i'\subset C_i$ we see that
$$\BS^3\setminus X=\cup B_i$$
is the union of nested balls $B_1\subset B_2\subset B_3\subset\dots$ It follows that $\BS^3\setminus X$ is homeomorphic to $\BR^3$. In particular $\BR^3\setminus X$ is homeomorphic to $\BS^2\times\BR$ and thus tame.

Continuing with the proof of Lemma \ref{lem-tucker} we can can thus assume that $C_i\setminus X$ is irreducible for all $i$. Suppose that we are given a compact manifold $C\subset C_1\setminus X$ and note that there is $j$ such that $C\cap C_j=\emptyset$. By assumption the fundamental group of $C_j\setminus X$ is finitely generated. The group $\pi_1(C_1\setminus(C\cup C_j))$ is also finitely generated because it is the fundamental group of a compact manifold. It follows thus from the Seifert-van Kampen theorem that $\pi_1(C_1\setminus(X\setminus C))$ is also finitely generated. Since $C\subset C_1\setminus X$ was arbitrary, we get from Tucker's theorem that $C_1\setminus X$ is tame. Since $\BR^3\setminus C_1$ is also tame, we get thus $\BR^3\setminus X$ is tame, as we wanted to prove.
\end{proof}

We will obtain the conclusion of Theorem \ref{sat1} by applying Lemma \ref{lem-tucker}. We will derive that the involved fundamental groups are finitely generated using the following direct consequence of the Seifert-van Kampen theorem:

\begin{lem}\label{lem-svk}
Let $M$ be a connected manifold and $\Sigma\subset M$ a properly embedded codimension 1 submanifold. If $\Sigma$ has finitely many connected components and if $\pi_1(U)$ is finitely generated for every such connected component $U\in\pi_0(M\setminus\Sigma)$, then $\pi_1(M)$ is finitely generated.\qed
\end{lem}

Continuing with our topological considerations note that descriptions of sets $X\subset\BR^3$ in terms of defining sequences are very well suited to check if Theorem \ref{sat1} applies to them or not. For instance we have:

\begin{prop}\label{prop-nottame}
Suppose that a closed connected set $X\subset\BR^3$ with connected complement admits a defining sequence $(C_i)$ such that
\begin{enumerate}
\item there is some $g$ such that $C_i$ is a handlebody of genus $g$ for all $i$,
\item the inclusion $\D C_i\to C_i$ is $\pi_1$-injective for all $i$, and
\item $C_i\setminus C_{i-1}\neq\D C_i\times[0,1)$ for all $i$.
\end{enumerate}
Then $\dim_\BR H_1(\BR^3\setminus X;\BR)\le g$ and $X$ is not tame.
\end{prop}
\begin{proof}
Proposition \ref{prop-nottame} is known to all experts in 3-dimensional topology - we just breeze over its proof. The third condition ensures that the inclusions of $\D C_i$ and $\D C_j$ into $C_i\setminus C_j$ are not isomorphisms on $\pi_1$ for all $i<j$. It then follows, using the second condition as well, that the fundamental group of $C_1\setminus X$ can be described as a non-trivial amagamated product with infinitely many factors:
$$\pi_1(C_1\setminus X)=\pi_1(C_1\setminus C_2)*_{\pi_1(\D C_2)}\pi_1(C_2\setminus C_3)*_{\pi_1(\D C_3)}\dots$$
As such, $\pi_1(C_1\setminus X)$ is not finitely generated. It follows thus from Lemma \ref{lem-tucker} that $\BR^3\setminus X$ is not tame.

Now, since homology commutes with direct limits we get that
$$H_1(\BR^3\setminus X;\BR)=\varinjlim H_1(\BR^3\setminus C_i;\BR).$$
Since $C_i$ is a genus $g$ handlebody for all $i$ we get that $H_1(\BR^3\setminus C_i;\BR)\simeq\BR^g$ for all $i$. This means that, $H_1(\BR^3\setminus X;\BR)$ is the direct limit of a sequence of vector spaces of dimension $g$ and hence has itself at most dimension $g$.
\end{proof}

Note that in the proof of the last claim of Proposition \ref{prop-nottame} it is indeed important that we are working over a field. Indeed, in the same situation $H_1(\BR^3\setminus X;\BZ)$ can well be infinitely generated. This happens for example if $X$ is a solenoid. In fact, the first \v{C}ech cohomology group with integral coefficients $\check H^1(X;\BZ)$ of such a solenoid $X$ is an infinitely generated abelian group contained in $\BQ$. By Alexander duality we have then that 
$$H_1(\BR^3\setminus X;\BZ)\simeq\check H^1(X;\BZ)$$
for any embedding of $X$ into $\BR^3$. Note that this implies that $\BR^3\setminus X$ can't be tame. Anyways, after referring to \cite{Conner-Meilstrup-Repovc} and the references therein for more on the topology of solenoid complements, and to  the classical book \cite{Eilenberg-Steenrod} for a discussion of \v{C}ech cohomology and Alexander duality, we record the upshot of the discussion above:

\begin{lem}\label{lem-mladen}
If $X\subset\BR^3$ is a solenoid then $H_1(\BR^3\setminus X;\BR)\simeq\BR$ and $\BR^3\setminus X$ is not tame.\qed
\end{lem}

\section{}
Besides the topological facts from the last section we will need below two further easy technical lemmas - we discuss both of them here. Suppose from now on that notation is as in the statement of Theorem \ref{sat1}.
\medskip

\noindent{\bf Smoothness.} The following easy lemma will play a key role in the proof of Theorem \ref{sat1}:

\begin{lem}\label{lem-constant}
The restriction of $f$ to $X$ is constant.
\end{lem}
\begin{proof}
The set of critical values of $f$ has measure $0$ by Sard's theorem. In particular it is totally disconnected. Since the image under $f$ of $X$ is connected and contained in the set of critical values, it is reduced to a point, as we needed to prove. 
\end{proof}

\begin{bem}
Before moving on, we comment briefly on the required smoothness for the function in Theorem \ref{sat1}. First, it suffices that $f\in C^3(\BR^3)$ because this is the differentiability assumption in Sard's theorem. Sard's theorem is only used in the proof of Lemma \ref{lem-constant}. In fact, Lemma \ref{lem-constant} does not hold with a lower degree of smoothness. Recall for instance that  Whitney \cite{Whitney} constructed a $C^1$-function on $f:\BR^2\to\BR$ for which there is a circle $S\simeq\BS^1$ such that $S\subset\{df=0\}$ while $f\vert_S$ is not constant. Note that such a circle cannot be rectifiable. In fact, if we assume that $X$ is rectifiable then, replacing in the proof of Lemma \ref{lem-constant} Sard's theorem by the fundamental theorem of calculus, we obtain the following second version of the main result of this note: {\em Let $X\subset\BR^3$ be as in Theorem \ref{sat1} and suppose moreover that it contains a dense connected rectifiable subset. If there is a function $f\in C^1(\BR^3)$ with  $X=\{x\in\BR^3\vert df_x=0\}$, then the open 3-manifold $\BR^3\setminus X$ is tame.}
\end{bem}

We will apply Morse theory to the function $f$. In fact we will need to work with manifolds with boundary, which is a bit of a pain. However, we will only need the simplest of the lemmas in Morse theory (compare with \cite[Theorem 3.1]{Milnor}). We discuss this for the sake of completeness. 
\medskip

\noindent{\bf Baby Morse theory.} Let $M$ be a compact manifold with boundary and recall that a function $F:M\to\BR$ is smooth if there is an open manifold $U$, with $M\subset U$ and with $\dim U=\dim M$ and such that our function $F$ extends to a $C^\infty$-function on $U$ which we still denote by $F$. Fix a Riemannian metric on $U$ and denote by $\DD F(x)$ the gradient of $F$ at $x$. 

\begin{lem}\label{lem-baby morse}
Suppose that $M$ is a compact manifold with boundary and $F:M\to\BR$ a smooth function. Suppose also that $t<T$ and that we have
\begin{enumerate}
\item $\DD F(x)\neq 0$ for all $x\in F^{-1}[t,T]$, and
\item $\DD F(x)$ is not orthogonal to $T_x\D M$ for all $x\in F^{-1}[t,T]\cap\D M$.
\end{enumerate}
Then the inclusion of $F^{-1}[T,\infty)$ into $F^{-1}[t,\infty)$ is a homotopy equivalence.
\end{lem}

\begin{proof}
To begin with, choose an open extension $U$ of $M$ to which $F$ extends. Also, given $x\in\D M$, decompose the gradient of $F$ as $\DD F(x)=\mu_x+\zeta_x$ with $\mu_x\in T_x\D M$ and $\zeta_x\in(T_x\D M)^\perp$ and note that our second assumption implies that 
$$\vert\zeta_x\vert<\vert\DD F(x)\vert$$
for all $x\in\D M\cap F^{-1}[t,T]$. Now, let $Z$ be a vector field with compact support, with $Z_x=\zeta_x$ for all $x\in\D M$ and with 
\begin{equation}\label{eq-trieste}
\vert Z_x\vert<\vert\DD F(x)\vert
\end{equation} 
for all $x\in f^{-1}[t,T]$. This is possible because by the first assumption we have that $\DD F(x)\neq 0$ for all $x$ in the region we are interested in. Now, consider a vector field $X$ on $U$, with compact support and with 
$$X_x=\frac1{\langle\DD F(x)-Z_x,\DD F(x)\rangle}\left(\DD F(x)-Z_x\right)$$
for all $x\in F^{-1}[t,T]$. Note that the denominator does not vanish because of \eqref{eq-trieste}. Note also that $X$ is tangent to $\D M$ for every point in $\D M\cap F^{-1}[t,T]$. And finally, note that by construction
$$\langle X_x,\DD F(x)\rangle=1$$
for all $x\in F^{-1}[t,T]$. All of this implies, exactly as in the proof of Theorem 3.1 in \cite{Milnor}, that the flow $(\phi_s)$ associated to $X_x$ satisfies 
$$\phi_{s-t}(F^{-1}[t,\infty))=F^{-1}[s,\infty)$$
for every $s\in[0,T-t]$. The claim follows.
\end{proof}

\section{}

We are now ready to prove Theorem \ref{sat1}. Starting with the proof suppose that there is a smooth function $f:\BR^3\to\BR$ whose set of critical points is exactly $X$, that is $X=\{df=0\}$. By Lemma \ref{lem-constant} we have that $f\vert_X$ is constant. In particular, we can normalize $f$ so that it vanishes on $X$, that is $X\subset f^{-1}(0)$. 

We claim that $\BR^3\setminus X$ is tame. From Lemma \ref{lem-tucker} we get that to prove that this is the case it suffices to show that $\pi_1(N\setminus X)$ is finitely generated for every compact submanifold $N\subset\BR^3$ containing $X$ in its interior. Fix such a compact submanifold $N\subset\BR^3$. We note that, since $\D N$ is compact and disjoint of the set $X$ of critical points of $f$ we can, up to a small isotopy supported near the boundary, assume that $\D N$ is in general position with respect to the vectorfield $\DD f$. In particular, there are only finitely many points $x\in\D N$ with $\DD f\in(T_x\D N)^\perp$. Moreover we can assume, again up to a small isotopy,  that none of those points belongs to the compact $f^{-1}(0)\cap\D N$. Scaling $f$ we can thus assume that 
\begin{equation}\label{eq-trieste2}
\DD f(x)\notin(T_x\D N)^\perp\text{ for all }x\in\D N\cap f^{-1}[-1,1].
\end{equation}
Set also $N_X=N\setminus X$. 

To prove that $\pi_1(N_X)$ is finitely generated we will apply Lemma \ref{lem-svk} to the surface $\Sigma=f^{-1}(0)\cap N_X$. We need to check that the conditions in the said lemma are satisfied. We start analysing the structure of $\Sigma$:
\medskip

\noindent{\bf Claim 1.} {\em $N_X\setminus\Sigma$ has finitely many connected components.}
\begin{proof}[Proof of Claim 1.]
Note that $\D N\setminus\Sigma$ has only finitely many connected components because $\Sigma\cap\D N$ is a closed submanifold of the compact surface $\D N$. In particular, it suffices to prove that every connected component $U$ of $N_X\setminus\Sigma$ meets $\D N_X$. Suppose that this is not the case for some component $U$ and let $\bar U$ be the closure of $U$ in $\BR^3$. The boundary of $\bar U$ is contained in $\Sigma\cup X=f^{-1}(0)$. In particular, since $\bar U$ is compact, it follows that $f$ has a critical point in $U$, which contradicts the assumption that $X=\{df=0\}$.
\end{proof}

We prove next that $\Sigma$ itself has finitely many connected components:
\medskip

\noindent{\bf Claim 2.} {\em $\Sigma$ has finitely many connected components.}
\begin{proof}[Proof of Claim 2.]
By assumption $\BR^3\setminus X$ has finite dimensional first homology group, and thus also finite dimensional first cohomology. Now, from the Mayer-Vietoris theorem, and from the fact both $H^1(\D N;\BR)$ and $H^1(\BR^3\setminus X;\BR)$ are finitely dimensional, we obtain then that
\begin{equation}\label{eq-home}
\dim_\BR H^1(N_X;\BR)<\infty.
\end{equation}
Now, note that every connected component $\Sigma'$ of $\Sigma$ determines, once endowed with an orientation, a cohomology class $\omega_{\Sigma'}\in H^1(N\setminus X;\BR)$ by taking algebraic intersection numbers. In other words, $\omega_{\Sigma'}$ is the Poincar\'e dual of $\Sigma'$.

Arguing by contradiction, suppose that $\Sigma=f^{-1}(0)\cap N\setminus X$ has infinitely many connected components. It then follows from Claim 1 that we can find an infinite sequence $\Sigma_0,\Sigma_1,\Sigma_2,\dots\in\pi_0(\Sigma)$ such that there are two (possibly identical) connected components $U$ and $V$ of $N_X\setminus\Sigma$ such that all the components $\Sigma_i$ in the list above are adjacent to both $U$ and $V$. For the sake of concreteness we will assume that $U\neq V$, leaving the other case to the reader. 

We proceed now as follows:
\begin{itemize}
\item Orient each $\Sigma_i$ in such a way that a positively oriented transversal segment crosses from $U$ to $V$, 
\item choose points $*_U\in U$ and $*_V\in V$, 
\item and choose also for each $i$ an arc $\tau_i$ from $*_U$ to $*_V$ and which meets $\Sigma$ transversally in a single point belonging to $\Sigma_i$.
\end{itemize}
For each $i\ge 1$ consider the 1-chain $\alpha_i=\tau_i-\tau_0$. By construction we have
$$\omega_{\Sigma_i}(\alpha_j)=\delta_{i,j}$$
for all $i,j\ge 1$ - here $\delta$ is the Kronecker function. This proves that the cohomology classes $\omega_{\Sigma_j}$ are linearly independent, contradicting \eqref{eq-home}. We have proved Claim 2.
\end{proof}

We have checked the first condition in Lemma \ref{lem-svk}. We prove now that the second condition is also satisfied:
\medskip

\noindent{\bf Claim 3.} {\em $\pi_1(U)$ is finitely generated for every component $U$ of $N_X\setminus\Sigma$.}

\begin{proof}[Proof of Claim 3]
Note that $f$ is either positive or negative on $U$. For the sake of having a better attitude, we assume that it is positive and set $U_T=\overline{U\cap f^{-1}[T,\infty)}$ for $T>0$. Note that \eqref{eq-trieste2} implies that $U_T$ is a compact manifold with boundary for all $T\in(0,1)$. In particular, $\pi_1(U_T)$ is finitely generated. Moreover, \eqref{eq-trieste2} and Lemma \ref{lem-baby morse} imply that for all $t\in(0,T]$ the inclusion of $U_T$ into $U_t$ is a homotopy equivalence. It follows that the inclusion of $U_T$ into $\cup_{t\in(0,T)}U_t=U$ is also a homotopy equivalence. In particular, $\pi_1(U)$ is finitely generated, as we had claimed.
\end{proof}

Having proved Claim 3, we see that also the second condition in Lemma \ref{lem-svk} is satisfied. It thus follows that $\pi_1(N_X)$ is finitely generated. As we mentioned at the beginning of the proof, it follows then from Lemma \ref{lem-tucker} that $\BR^3\setminus X$ is tame. This finishes the proof of Theorem \ref{sat1}.\qed
\medskip

Note that the proof of Theorem \ref{sat1} also gives some information about what is going on in higher dimensions. More concretely, the argument we used to prove that $\pi_1(N\setminus X)$ is finitely generated remains valid in that setting. In particular, we have the following weaker version of Theorem \ref{sat1}:

\begin{sat}
Let $X\subset\BR^n$ be a closed and connected set whose complement $\BR^n\setminus X$ is connected and satisfies $\dim_\BR H_1(\BR^n\setminus X;\BR)<\infty.$ If $X$ is critical, then $\pi_1(N\setminus X)$ is finitely generated for every compact $n$-dimensional submanifold $N\subset\BR^n$ which contains $X$ in its interior.\qed
\end{sat}

\end{document}